\documentclass{amsart}

\usepackage{amsmath}
\usepackage{amsthm}
\usepackage{amssymb} 

\newtheorem{thm}{Theorem}[section]

\newtheorem{prop}[thm]{Proposition}

\newtheorem{cor}[thm]{Corollary}

\theoremstyle{definition}
\newtheorem{defn}[thm]{Definition}
\newtheorem{rem}[thm]{Remark}

\DeclareMathAlphabet{\mathbbb}{U}{bbold}{m}{n}
\DeclareMathAlphabet{\mathmanual}{U}{manfnt}{m}{n}

\newcommand{\calA}{\mathcal{A}} 
 
 \newcommand{\I}{\mathbbb{1}}

\newcommand{\Sim}{\Sigma}

\newcommand{\Hom}{\mathrm{Hom}}
\newcommand{\End}{\mathop{\mathrm{End}}}

\newcommand{\Q}{\mathbb{Q}}

\newcommand{\tens}{\otimes}

\def\c#1{\mathop{ {\mathcal #1}  }\nolimits}

\def\ch{\mathop{\mathcal Ch}\nolimits}
\def\cch{\mathop{\mathcal CCh}\nolimits}

\def\ch-{\mathop{{\mathcal Ch}_-}\nolimits}
\def\cch-{\mathop{{\mathcal CCh}_-}\nolimits}
\def\ch+{\mathop{{\mathcal Ch}_+}\nolimits}
\def\cch+{\mathop{{\mathcal CCh}_+}\nolimits}

\def\Tr{{\mathsf{Tr} }}

\def\cp{{\mathsf{cp} }}

\def\kim{{\mathrm{kim}}}

\let\sem=\bf

\def\ZZ{{\mathbb Z}}
\def\QQ{{\mathbb Q}}

\def\fraz#1 #2 {\frac{#1}{#2}}

\def\h#1{{\kern .07em}^h\kern-.04em{#1}}
\def\a#1{{\kern .07em}{#1}_a}

\let\sem=\bf

\def\To{\longrightarrow}

\numberwithin{equation}{thm}

\begin{document}
\title[Schur-finiteness and endomorphisms universally...]{Schur-finiteness and endomorphisms universally of trace zero
via certain trace relations}

\date{October 2nd, 2007}

\begin{abstract}
We provide a sufficient condition that ensures the nilpotency 
of endomorphisms universally of trace zero of Schur-finite objects
in a category of homological type, i.e., a $\Q$-linear $\tens$-category
with a tensor functor to super vector spaces. 
This generalizes previous results about finite-dimensional objects,
in particular by Kimura in the category of motives.
We also present some facts which suggest that this might be the best
generalization possible of this line of proof.
To get the result we prove an identity
of trace relations on super vector spaces which
  has an independent interest in the field of combinatorics. Our main tool is Berele-Regev's
  theory of Hook Schur functions. We use their generalization of the
  classic Schur-Weyl duality
to the ``super'' case, together with their 
factorization formula.
\end{abstract}

\author{Alessio Del Padrone}
\email{delpadro@dima.unige.it}

\author{Carlo Mazza}
\email{mazza@dima.unige.it}
\maketitle


%







\section{Introduction}

Let $\calA$ be a $\Q$-linear tensor category 
in which idempotents split equipped with a 
functor $H$ to super vector spaces. The endomorphisms
universally of trace zero of an object are the endomorphisms whose
compositions with any other endomorphisms have all trace zero.
In the case $\calA$ is a category of motives, the endomorphisms
$\mathcal{N}(A)$ universally of trace zero of an object $A$ are a
subset of the numerically trivial ones. According to a result by
Kimura in \cite{KiFD}, if an object $A$ is finite-dimensional, then
every numerically trivial endomorphism of $A$ is nilpotent. A still
open question is whether the same result holds for Schur-finite
objects (see \cite{DeMaSFN}).

Finiteness conditions for motives are related to part of the standard
conjectures: in particular, if $\calA$ is the category of Chow
motives, then the finiteness of the motive of a surface with $p_g=0$
is equivalent to Bloch's Conjecture (see \cite[Theorem 7]{GuPeFD}). In
this paper we show (Theorem \ref{ppp}) that if $A$ has the sign
property (Definition \ref{sign}) and $S_\lambda(A)=0$, where $\lambda$
is a partition which is not too big, i.e., it does not contain the
rectangle with $a+2$ rows and $b+2$ columns, $a$ and $b$ being the
dimensions of the even and the odd part of $H(A)$, then every
endomorphism in $\mathcal{N}(A)$ is nilpotent.

We start by recalling the definitions
for the different finiteness notions and their main properties.
We then recall the the nilpotency conjecture and how this relates
to the various finiteness notions.
Then the main result is stated and proved (modulo a combinatorial result)
and we dedicate the rest of the section
to analyzing some related facts and reasons why this might be the sharpest 
result possible using this particular line of proof.
The second and last section of the paper will prove the necessary combinatorial facts
which are used in the proof of the main result.


We keep the notation and the terminology of \cite{DeMaSFN},
except that we will write $\Tr$ for the categorical trace in the sense of \cite{JSV}, 
$tr$ for the ordinary trace of matrices, and we will write $V_0|V_1$ ($d_0|d_1$) for (the dimension of) the super vector space having the $d_i$-dimensional vector space $V_i$ in degree $i$. 


Since neither of the authors is an expert in combinatorics, the third
section will be less concise and we now recall some basic facts and
notations: a partition $\lambda$ of $n$ is a sequence of integers
$(\lambda_1,\ldots,\lambda_r)$ such that $\lambda_i\geq \lambda_{i+1}>0$
for all $i=1,\ldots,r-1$ and $\sum_i \lambda_i=n$. We will often
confuse a partition with its associated Young diagram and, e.g.,
we say that the partition $(b^a)$ is the rectangle with $a$
rows and $b$ columns. If $\lambda$ is a partition,
we write $\lambda'=(\lambda_1',\ldots,\lambda_s')$ for the
transposed partition. We say that $(i,j)\in \lambda$ if $\lambda_i\geq
j$. The maximal hook of $\lambda$ is the hook $(\lambda_1,1^{r-1})$;
the maximal skew hook is the set $\{(i,j)\in\lambda$ s.t.
$(i+1,j+1)\not\in\lambda\}$. If $\nu$ is the maximal (skew or not)
hook, then
$\lambda\setminus\nu$ is the partition
$(\lambda_2-1,\ldots,\lambda_r-1)$.  Let $\mu=(\mu_1,\ldots,\mu_s)$
and $\lambda=(\lambda_1,\ldots,\lambda_r)$ be two partitions, we say
that $\mu\subseteq \lambda$ if $s\leq r$ and $\mu_i\leq \lambda_i$ for
all $i=1,\ldots,s$.



\section{Finite-dimensionality and nilpotency}

Let $\mathcal{A}$ be a {\sem pseudo-abelian $\tens$-category},
i.e., a ``$\tens$-cat\'egorie rigide sur $F$''
as in \cite[2.2.2]{AnM}
in which idempotents split. We assume that $F=\End_{\c{A}}(\I)$ and it
contains $\QQ$. 
The partitions $\lambda$ of an integer $n$ give a complete
set of mutually orthogonal central idempotents 
\[ \mathsf{d}_\lambda:=\frac{\dim
  V_\lambda}{n!}\sum_{\sigma\in \Sigma_n} \chi_\lambda(\sigma)\sigma\]
in the group algebra $\Q\Sim_n$ (see \cite{FH}).  We define an
endofunctor on $\c{A}$ by setting $S_\lambda(A)=\mathsf{d}_\lambda(A^{\tens
  n})$. This is a multiple of the classical Schur functor
corresponding to $\lambda$. In particular, we define
$\mathrm{Sym}^n(A)=S_{(n)}(A)$ and $\Lambda^n(A)=S_{(1^n)}(A)$.  The
following definitions are directly inspired by \cite{DeCT02} and
\cite{KiFD} (see \cite{AK}, \cite{GuPeFD}, and \cite{MaSFM} for
further reference).

\begin{defn}
  An object $A$ of $\calA$ is \textbf{Schur-finite} if there is a
  partition $\lambda$ such that $S_\lambda(A)=0$.  If $S_\lambda(A)=0$
  with $\lambda$ of the form $(n)$ (respectively, $\lambda=(1^n)$)
  then $A$ is called \textbf{odd} (respectively, \textbf{even}).  We
  say that $A$ is \textbf{finite-dimensional} (in the sense of
  Kimura-O'Sullivan) if $A=A_+\oplus A_-$ with $A_+$ even and $A_-$
  odd.
\end{defn}

Both finite-dimensionality and Schur-finiteness are stable under
direct sums, tensor products, duals, and taking direct summands.
Hence every finite-dimensional object is Schur-finite, but the converse 
does not hold (see \cite[2.6.5.1]{DePhd}). On the other hand, this weaker
condition is compatible with triangulated structures on the
category while finite-dimensionality is not (see \cite[3.6 and 3.8]{MaSFM}).

One of the most important consequences of finite-dimensionality
is the nilpotency of endomorphisms universally of trace zero.

\begin{defn}
Recall that we have
$F$-linear {\sem trace} maps $\Tr\colon \End_{\c{A}}(A)\To \End_{\c{A}}(\I)$
compatible with $\otimes$-functors. We define
the $F$-submodules of {\sem endomorphisms universally of trace zero} as
\[
\c{N}(A):=\{f\in\Hom_{\c{A}}(A,A)\mid \Tr(f\circ g)=0,\;\; {\rm for\;\;all}\;\;
g\in\Hom_{\c{A}}(A,A)\}.\]
We say an object $A$ is a \textbf{phantom} if $\mathrm{Id}_A\in \c{N}(A)$.
\end{defn}


Andr\'e and Kahn proved in 
\cite[9.1.14]{AK} that if $A$ is a finite-dimensional
object, then any $f\in\mathcal{N}(A)$ is nilpotent.
In particular, if all objects
of $\calA$ are finite-dimensional, then 
the projection functor $\calA\to \calA/\mathcal{N}$ lifts idempotents and 
is conservative (hence ``there are no phantom objects'').
In general $\tens$-categories, Schur-finiteness
is not sufficient to get the nilpotency of $\mathcal{N}(A)$;
see \cite[10.1.1]{AK} for an example of a phantom non-zero Schur-finite 
object, i.e., whose identity is universally of trace zero.

In the special case of Chow motives, if $M(X)$
is the motive of a smooth variety $X$,
then $\mathcal{N}(M(X))$ is the set of the numerically 
trivial correspondences $CH^{\dim X}(X\times X)_{num}$,
and the Andr\'e-Kahn result generalizes a previous result
by Kimura (\cite[7.5]{KiFD}) who
also conjectured in \emph{loc.~cit.} that
all Chow motives are finite-dimensional, and hence
all $\mathcal{N}(M(X))$ are nilpotent.
Moreover, the conjectures of Bloch-Beilinson-Murre (together with the ``numerical=homological'' standard conjecture) imply
the nilpotency of all endomorphism algebras
 and this implies the finite-dimensionality of each object. 
In order to extend the Andr\'e-Kahn result to a larger subclass 
of Schur-finite objects, we needed to find
a peculiar feature which forces the nilpotency and is expected to be true
in the category of motives. We will show that the
sign property is such a feature.

From now on, let $\calA$ be a category of \textbf{homological type}
(see \cite[4]{KaMM}), i.e., a category
with a $\tens$-functor to super vector
spaces $H:\calA\to sVect$ which we will call ``cohomology'' by abuse of notation.

\begin{defn}\label{sign}
We say that an object $A$ in a category of homological type
has the {\bfseries sign property} if 
the projections on the even and the odd part of the cohomology 
$H(A)=H(A)_0\oplus H(A)_1$ lift to endomorphisms in $\calA$ 
(cf. \cite[4.8]{KaMM}).
\end{defn}

In particular, the category of motives is of homological type
and the sign property is known as the sign conjecture
(\cite[p. 426]{JaFDMMC}) which is a part of the conjecture on 
the algebraicity of the Chow-K\"unneth decomposition of the diagonal.
The main difference is that we do not require the lifts to be
idempotents or orthogonal, although they are so in cohomology.

The next theorem is our main result: its proof
relies on a combinatorial result which will be
proved in \S\ref{comb}.
The rest of this section will be
dedicated to some related remarks some of which suggest 
that this might the best generalization possible of this line of proof.

\begin{thm}\label{ppp}
Suppose $A$ has the sign property and 
let $H(A)$ be of dimension $d_0|d_1$.
Let $S_\lambda(A)=0$ for a partition $\lambda$ of $n\geq 2$ such that 
$(d_1+2)^{(d_0+2)}\nsubseteq \lambda$ and let $s$ be the length
of the biggest hook in $\lambda$. Then for any $f\in \mathcal{N}(A)$ we have $f^{\circ (s-1)}=0$.
\end{thm}

\begin{proof}
Let $\nu$ be the maximal hook $\nu$ of $\lambda$, $r:=n-s$ and $\delta:=\lambda
\setminus \nu$. A key role is played by the following function:
for any $f_1,\ldots,f_r\in \End(A)$, let
\begin{multline}\label{defy}
y(\delta; f_1,\dots, f_r):=\Tr(\mathsf{d}_\delta\circ f_1\otimes\cdots\otimes f_r)\\
=\frac{\dim V_\delta}{|\delta|!}\sum_{\sigma\in \Sigma_r}
\chi_\delta(\sigma)\prod_{j=1}^q\Tr(f_{{\gamma_j}^{l_j-1}(k_j)}\circ \cdots \circ
f_{\gamma_j(k_j)}\circ f_{k_j}),
\end{multline}
where $\gamma_1\circ\cdots\circ\gamma_{q_\sigma}$ is the cycle decomposition of $\sigma$,
$k_j$ is any element in the support of $\gamma_j$,
and $l_j$ is the length of the cycle $\gamma_j$.
Our proof of \cite[Theorem 2.1]{DeMaSFN} shows that if $y(\delta;-,\ldots,-)$ is not the zero function on $\End(A)^r$, then $f^{\circ(s-1)}=0$
for each $f\in \c{N}(A)$.

Since $A$ has the sign property then there exist two endomorphisms $\pi_0$ and $\pi_1$
such that $t_i:=\Tr(H(\pi_i))=(-1)^id_i$ ($i=0,1$), and
we have the following trace identities:
\begin{enumerate}
\item[\textbf{(TI1)}] $t_i= \Tr(\pi_i) = \Tr(\pi_i^l)\in \ZZ$ for all  $l>0$, and
\item[\textbf{(TI2)}] $\Tr(\pi_{j_1}\circ\cdots \circ \pi_{j_k}) = 0$ for any 
$k>1$ and any non-constant map $j\colon\{1,...,k\}\To\{0,1\}$.
\end{enumerate}  

Let us choose
$f_1=\dots=f_r=g:=\alpha_0\pi_0+\alpha_1\pi_1$, then
it is an immediate consequence of the properties
$(\mathbf{TI1})$ and $(\mathbf{TI2})$ that
$y$ becomes a polynomial in $\alpha_0,\alpha_1,t_0,t_1$
\begin{equation*}
y(\delta;g):=
y(\delta;g,\dots,g)=
\frac{\dim V_\delta}{|\delta|!}\sum_{\sigma\in \Sigma_{|\delta|}}
\chi_\delta(\sigma)\prod_{j=1}^q({\alpha_0}^{l_j}t_0+{\alpha_1}^{l_j}t_1)
\end{equation*}
Since $H$ is a tensor functor, $S_\lambda(H(A))=H(S_\lambda(A))$.
But $S_\lambda(H(A))=0$ if and only if $(d_0+1,d_1+1)\in \lambda$
(see \cite[1.9]{DeCT02}) 
and therefore $S_\lambda(A)=0$ implies that $(d_0+1,d_1+1)
\in \lambda$, and, in particular, $(d_0,d_1)\in\delta$.
But by hypothesis, we have that
$(d_0+2,d_1+2)\not\in\lambda$ and then 
$(d_0+1,d_1+1)\not\in\delta$.
So $(d_0,d_1)$ is in the maximal skew hook of $\delta$.

From the results in \S\ref{comb}, it follows
that $y(\delta;g)$ is the polynomial 
$P(\delta;{\alpha_0},{\alpha_1};t_0,t_1)$
in $\Q[\alpha_0,\alpha_1,t_0,t_1]$ which,
when computed for $t_0=d_0$ and $t_1=-d_1$, is a non-zero
polynomial in $\alpha_0$ and $\alpha_1$.
Since the coefficients of this polynomial are a field of characteristic
zero, this proves the theorem.
\end{proof}



\begin{rem}\label{dnd}
Proving that the function 
$y(\lambda\setminus\nu;-,\ldots,-)$ (defined in equation
\ref{defy}) is not zero 
is a combinatorial
problem because it
does not depend on the choice of the category. In particular,
it can be calculated on a super-vector space.
\end{rem}

\begin{rem}
If $S_\lambda(A)=0$ and its cohomology is of super dimension
$d_0|d_1$, then $\lambda\supseteq ((d_1+1)^{d_0+1})$, and there
exist $f_1,\ldots,f_r$ such that 
$y(\lambda\setminus \nu; f_1,\dots, f_r)\not=0$ only if
$\lambda\nsupseteq ((d_1+2)^{d_0+2})$.

The first assertion is in the proof of \ref{ppp}, while 
the second one follows from 
the universal relations among super traces
characterized by Razmyslov. A suitable formulation 
of his result is in Berele's \cite[3.1]{BeTI}
where it is shown that if $(d_0+2,d_1+2)\in\lambda$,
and so $(d_0+1,d_1+1)\in\delta:=\lambda\setminus \nu$, 
then
${\sf d}_\delta\in\QQ\Sigma_{|\delta|}$ belongs to the two sided ideal
of $\QQ\Sigma_{|\delta|}$ of those elements whose associated ``trace polynomial''
(in our notation, just a multiple of $y(\delta;-,\ldots,-)$ 
by a non-zero constant)
is an identity of super matrices with even dimension $d_0$ and odd dimension
$d_1$.
That is, if $\lambda\supseteq ((d_1+2)^{d_0+2})$, then
$y(\lambda\setminus \nu; f_1,\dots, f_r)=0$ for any choice of $f_1,\ldots,f_r$. 
For more details, see the given reference, or \cite{DeSFSDD}, or \cite[2.5]{DePhd}.
\end{rem}

\begin{rem}\label{RemarkP}
Note that in categories of homological type, and in particular for motives,
there is really nothing lost in considering only $\End_{\c{A}}(A)^r$ instead
of trying to exploit $\End_{\c{A}}(A^{\otimes r})$, for
$\End(H(A^{\otimes r}))\cong\End(H(A))^{\otimes r}$ as vector spaces.
\end{rem}

\begin{rem}
(P. Deligne) By \ref{dnd},
it suffices to do the calculations on a super vector space $V$.
Letting $\pi_0$ and $\pi_1$ be the projections, $\pi_0-\pi_1$ is $(-1)^i$ in degree $i$ and so
\[ \Tr(\mathsf{d}_\delta\circ (\pi_0-\pi_1)^{\tens r})=
\dim(Im(\mathsf{d}_\delta)^+)+\dim(Im(\mathsf{d}_\delta)^-)\]
which is non-zero if and only if $S_\delta(V)\not=0$. This is a sufficient condition
for the non-vanishing of $y(\delta;\pi_0-\pi_1)$.
\end{rem}

\begin{rem}
By \ref{ppp}, we have that $f^{\circ(s-1)}=0$ for all $f\in \mathcal{N}(A)$.
Using 
Razmislov's improvement of the (Dubnov-Ivanov) Nagata-Higman Theorem
(see \cite{FoNHT} and \cite[11.8.10]{PrLG}) we have that
$\c{N}(A)^{(s-1)^2}=0$. Thus
for finite-dimensional objects $A=A_+\oplus A_-$ 
the known nilpotency bounds (\cite[3.4]{AnMDF}) can be improved:
indeed such an object is (minimally) killed by the rectangle $\lambda=((\kim(A_-)+1)^{\kim(A_+)+1})$,
hence $s-1=\kim(A)$ (\cite[2.4.10]{DePhd} or \cite{DeSFSDD}).
In the setting of categories of homological type 
a different bound is given in \cite[4.10 b)]{KaMM} (see also \cite[4.11]{KaMM}).
However, for any (twist of direct summand of a) pure motive $\mathfrak{h}(X)$,
the global nilpotency bound for 
$\c{N}({\mathfrak{h}(X)})$ should be $\dim(X)+1$ by Bloch-Beilinson-Murre's
conjectures (\cite[Conjecture 2.1 (strong e)]{JaMS}).
As a partial evidence, Morihiko Saito proved in \cite{SaBCCM} that for smooth complex surfaces $S$ with $p_g=0$,
Bloch's conjecture (\cite[Conjecture 1.8]{JaMS}) is equivalent to $(CH^2(S\times S)_{hom})^3=0$.
\end{rem}

\section{The combinatorial result}\label{comb}


We now state and prove the main combinatorial result. As a corollary,
we will get the non-vanishing we needed. We are interested in studying
the following polynomial.

\begin{defn}\label{defp}
Let $\delta$ be a partition of $r$, we define:  
  \begin{equation*}
P(\delta;{\alpha_0},{\alpha_1};t_0,t_1):=
\frac{\dim V_\delta}{r!}\sum_{\sigma\in \Sigma_{r}}
\chi_\delta(\sigma)\prod_{j=1}^q({\alpha_0}^{l_j}t_0+{\alpha_1}^{l_j}t_1)
  \end{equation*}
as a polynomial in $\QQ[{\alpha_0},{\alpha_1},t_0,t_1]$,
where the $l_j$ are the lengths of
the cycles in the cycle decomposition of $\sigma$.
\end{defn}

Notice that $P(\delta;1,0;t_0,t_1)$ is the content
polynomial of $\delta$ as an element of $\Q[t_0]$.




\begin{prop}\label{prp}
Let $\delta$ be any partition, and let
$P(\delta;{\alpha_0},{\alpha_1};t_0,t_1)$ be
the polynomial defined in \ref{defp}.
Then for any $(d_0,d_1)$ in the maximal skew hook of $\delta$ we have
the following identity of polynomials in $\alpha_0$ and $\alpha_1$:
\begin{multline*}
 P(\delta;{\alpha_0},{\alpha_1};d_0,-d_1)\\
=(\dim V_\delta)(-1)^{|\nu|}\frac{\dim V_\mu} {|\mu|!}\frac{\dim V_\nu}{|\nu|!}(\alpha_0-\alpha_1)^{d_0d_1}{\alpha_0}^{|\mu|}{\alpha_1}^{|\nu|}\cp_\mu(d_0)\cp_\nu(d_1),
\end{multline*}
where $\mu$ is the partition whose non-zero parts are the positive $\delta_i-d_1$, and
$\nu$ is the partition whose non-zero parts are the positive 
$\delta_i^\prime-d_0$ (see \cite[6.14]{BeRe}).
\end{prop}

\begin{proof}
Let $\c{A}$ be the $\QQ$-linear
category of super vector spaces and let $A$ be 
$\QQ^{d_0}|\QQ^{d_1}$.
By Remark \ref{dnd},
we can describe the polynomial $P(\delta;{\alpha_0},{\alpha_1};d_0,-d_1)$
in terms of the (super) trace of $g\in \End_{\c{A}}(A)$ given by $\alpha_0\pi_0+\alpha_1\pi_1$,
where $\pi_i\in\End_{\c{A}}(A)$ is the projection on the degree $i$ part
(if $\alpha_0\alpha_1\neq 0$ then $g\in GL(d_0|d_1)$). More precisely,
if $e_\delta$ is any (idempotent) Young symmetrizer  associated  to $\delta$,
then $\overline{P}=\frac{1}{\dim V_\delta}P(\delta;{\alpha_0},{\alpha_1};d_0,-d_1)$
is the evaluation on $\widetilde{g}:=\alpha_0\pi_0-\alpha_1\pi_1$ of the character $\chi_\delta$
of the representation $e_\delta^A A^{\otimes |\delta|}$ of
$GL(d_0|d_1)$.
Indeed, letting ${\sf s}_\vartheta$ be the usual Schur function of a partition $\vartheta$ we have:
\begin{eqnarray*}
\overline{P}
(\delta;{\alpha_0},{\alpha_1};d_0,-d_1)
&=&\frac{1}{\dim V_\delta}\Tr(\mathsf{d}_\delta\circ g^{\otimes |\delta|})\\
&=&\Tr(e_\delta\circ g^{\otimes |\delta|})\\
&=&tr(e_\delta\circ \widetilde{g}^{\otimes |\delta|})\\
&=&tr(e_\delta\circ (\alpha_0\pi_0-\alpha_1\pi_1)^{\otimes |\delta|})\\
&\overset{(a)}{=}&
{\sf HS}_\delta(\underbrace{\alpha_0,\dots,\alpha_0}_{d_0},\underbrace{-\alpha_1,\cdots,-\alpha_1}_{d_1})\\
&\overset{(b)}{=}&
\left(\prod_{1}^{d_0}\prod_{1}^{d_1}(\alpha_0-\alpha_1)\right)
{\sf s}_\mu(\underbrace{\alpha_0,\dots,\alpha_0}_{d_0}){\sf s}_\nu(\underbrace{-\alpha_1,\dots,-\alpha_1}_{d_1})\\
&=&(-1)^{|\nu|}(\alpha_0-\alpha_1)^{d_0d_1}{\alpha_0}^{|\mu|}{\alpha_1}^{|\nu|} 
{\sf s}_\mu(\underbrace{1,\dots,1}_{d_0}){\sf s}_\nu(\underbrace{1,\dots,1}_{d_1}),
 \\
&\overset{(c)}{=}&(-1)^{|\nu|}\frac{\dim V_\mu} {|\mu|!}\frac{\dim V_\nu}{|\nu|!}(\alpha_0-\alpha_1)^{d_0d_1}{\alpha_0}^{|\mu|}{\alpha_1}^{|\nu|}\cp_\mu(d_0)\cp_\nu(d_1),
\end{eqnarray*}
where $\mu$ is the partition whose non-zero parts are the positive $\delta_i-d_1$, and
$\nu$ is the partition whose non-zero parts are the positive 
$\delta_i^\prime-d_0$ (see \cite[6.14]{BeRe}).

The hyphotesis $(d_0+1,d_1+1)\not\in\delta$ allows us to exploit
\cite[6.10 (b)]{BeRe} in $(a)$, and \cite[6.20]{BeRe} in $(b)$. 
Notice that the $\Sigma_{|\delta|}$ action defined in \cite[1]{BeRe}
coincides with the one induced by the (Koszul) commutativity constraint of $\c{A}$.
The equality $(c)$ is a well-known fact (see
\cite[I.3, Example 5 and the proof of I.7(7.6)]{Mac95}).
\end{proof}

\begin{cor}\label{cr}
Let $P$ be the polynomial defined in \ref{defp}. If $(d_0,d_1)$ 
is in the maximal skew hook of $\delta$,
then $P(\delta;\alpha_0,\alpha_1;d_0,-d_1)$
is a non-zero polynomial in $\alpha_0$
and $\alpha_1$.

\end{cor}
\begin{proof}
By Theorem \ref{prp}, we just need to show that 
$\textsf{cp}_\mu(d_0)\not=0$ and $\textsf{cp}_\nu(d_1)\not=0$.
Let us show the first one, since the second one comes from passing
to the transpose partition. 
If $(d_0,d_1)$ is in the maximal skew hook of $\delta$, then
$\delta_{d_0+1}\leq d_1$,
so $\mu$ has at most $d_0$ rows 
which in turn implies that $\textsf{cp}_\mu(d_0)\not=0$. 
\end{proof}

Shortly after the proof of the Theorem \ref{prp} was completed, 
Christine Bessenrodt informed us that Richard Stanley had suggested
a slightly different solution: it follows 
from the theory of super Schur functions, for our polynomial $P$ can be obtained
from the power-sum expansion of the Schur function ${\sf s}_\delta$
by a suitable substitution. Then, one can apply Problem $50(g)$ 
in the Supplementary Problems for Chapter $7$ of
Stanley's book Enumerative Combinatorics II (on his EC web page),
which is based on \cite{BeRe}.
We decided to maintain our approach for it is closer
to our original ``motivic'' point of view.

\section*{Acknowledgments}
We thank Christine Bessenrodt for her suggestions and interest in the problem,
and Bruno Kahn for supplying the most recent version of his work.
The first author wishes also to express his gratitude to Uwe Jannsen for his constant support and 
lots of useful conversations during a postdoc stay at the math department of the University of Regensburg.

\bibliographystyle{mrl}
\bibliography{crpaper2}

\end{document}